\documentclass[times,5p]{elsarticle}

\usepackage{lineno,hyperref}
\modulolinenumbers[1]

\journal{arxiv.org}

\usepackage[latin1]{inputenc}
\usepackage{url}
\usepackage{amssymb}
\usepackage{amsopn}
\usepackage{amstext}
\usepackage{amsmath}
\usepackage{todonotes}
\usepackage{amsthm}
\usepackage{mathdots}
\usepackage{algpseudocode}
\usepackage{algorithm}

\newtheorem{theorem}{Theorem}
\newtheorem{lemma}{Lemma}
\newtheorem{corollary}{Corollary}
\newtheorem{observation}{Observation}
\newtheorem{proposition}{Propositon}
\theoremstyle{definition}
\newtheorem{definition}{Definition}

\theoremstyle{remark}

\newtheorem{example}{Example}

\usepackage{mathabx}
\def\antisorted{\updownharpoons}
\def\notantisorted{\slash\hspace{-1.05em}\updownharpoons}
\newcommand{\R}{{\ensuremath\mathbf R}}
\newcommand{\Z}{{\ensuremath\mathbf Z}}
\newcommand{\Symm}{{\ensuremath\mathfrak S}}
\newcommand{\Oh}{{\ensuremath\mathcal O}}

\newcommand{\NP}{{\ensuremath\mathcal NP}}
\newcommand{\ZPP}{{\ensuremath\mathcal ZPP}}
\newcommand{\Prob}[1]{\mathrm{Pr}\!\left[{#1}\right]}
\newcommand{\dropped}[2]{#1^{[#2]}}
\DeclareMathOperator{\VaR}{VaR}
\DeclareMathOperator{\dist}{dist}
\DeclareMathOperator{\argmin}{argmin}

\begin{document}
\begin{frontmatter}

\title{Bounding Stochastic Dependence,\\
  Complete Mixability of Matrices, and\\
  Multidimensional Bottleneck Assignment Problems}

\author{Utz-Uwe Haus}
\address{Institute for Operations Research, Dept. of Mathematics, ETH
    Zurich, Switzerland, uhaus@ifor.math.ethz.ch}
\begin{abstract}
  We call a matrix completely mixable if the entries in its columns
  can be permuted so that all row sums are equal. If it is not
  completely mixable, we want to determine the smallest maximal and
  largest minimal row sum attainable. These values provide a discrete
  approximation of of minimum variance problems for discrete
  distributions, a problem motivated by the question how to estimate
  the $\alpha$-quantile of an aggregate random variable with unknown
  dependence structure given the marginals of the constituent random
  variables. We relate this problem to the multidimensional bottleneck
  assignment problem and show that there exists a polynomial
  $2$-approximation algorithm if the matrix has only $3$ columns. In
  general, deciding complete mixability is $\NP$-complete. In
  particular the swapping algorithm of Puccetti et
  al.~\cite{puccetti-rueschendorf:11} is not an exact method unless
  $\NP\subseteq\ZPP$. For a fixed number of columns it remains
  $\NP$-complete, but there exists a PTAS. The problem can be solved
  in pseudopolynomial time for a fixed number of rows, and even in
  polynomial time if all columns furthermore contain entries from the
  same multiset.
\end{abstract}


\end{frontmatter}

\section{Introduction}
\label{sec:introduction}

The problem we are considering is the following:
Given a matrix $A\in\R^{m\times d}$, we are interested in the best way
of permuting entries in each column (independently) so that the
maximal row sum is minimized, or so that the minimal row sum is
maximized. Given $d$ permutations $\Pi=(\pi_1,\dots,\pi_d)\in\Symm(m)^d$ we
denote by $A^\Pi$ the matrix obtained from $A$ by permuting column $j$
by $\pi_j$, i.e.~$A^\Pi_{i,j}=A_{\pi^{-1}_j(i),j}$. The optimization
problem is then

\begin{equation}
  \label{eq:gamma}
  \gamma(A):=\min_{\Pi\in\Symm(m)^d} \max_{1\leq i\leq m}\left\{\sum_{j=1}^d A^\Pi_{i,j}\right\}
\end{equation}
and
\begin{equation}
  \label{eq:beta}
  \beta(A):=\max_{\Pi\in\Symm(m)^d} \min_{1\leq i\leq m}\left\{\sum_{j=1}^d A^\Pi_{i,j}\right\}.
\end{equation}

We
note that aggregation operations other than $+$ are conceivable (e.g.,
$\min,\max,\times$), but will not be treated here.

This problem is motivated by an application in quantitative finance,
but in fact arises whenever one needs to estimate the influence of
stochastic dependence on a statistical problem: Consider an aggregate
random variable $L$ of the form $L=\sum_{i=1}^dL_i$, where the random
variables $L_i$ are possibly not independent. Denote by
$F_L(x)=P(L\leq x)$ the distribution function of $L$. We are
interested in computing the $\alpha$-quantile (Value-at-Risk, $\VaR_\alpha$)
$F^{-1}_L(\alpha)=\inf\{x\in\R\::\: F_L(x)\geq\alpha\}$, for
$\alpha\in(0,1)$. Often we have no data on the joint distribution $L$,
but only on the marginal distributions $F_j$ of the constituent random
variables $L_j$, and we also lack information on the dependence
structure between them. 

In the following we will assume that the marginal distributions are
discrete, or have been approximated from below and from above as
described in~\cite{puccetti-rueschendorf:11}: For $F_i$ the
generalized inverse is 
$F^{-1}_j(\alpha)=\sup\{x\in\R\::\: F_j(x)\leq \alpha\}$. Consider a
discretization in $N+1$ points. Compute the values $q^j_r=F^{-1}_j(r/N)$
for $r\in\{0,1,\dots,N\}$. Denoting by $1_{[a,b)}$ the characteristic function
on the interval $[a,b)$,
$$\underline{F_j}(x)=\frac{1}{N}\sum_{r=0}^{N-1}1_{[q_r^j,+\infty)}(x)
\text{ and }
\overline{F_j}(x)=\frac{1}{N}\sum_{r=1}^{N}1_{[q_r^j,+\infty)}(x),
$$
provide discrete approximations of $F_j$ with $\underline{F_j}\geq
F_j\geq\overline{F_j}$. 

Dependence among the individual $F_j$ will manifest itself in the
way the values $q^j_r=F^{-1}_j(r/N)$ are appearing in the matrix
$
A=\begin{pmatrix}
  q^1_0&\cdots&q^d_0\\
  \vdots &    &\vdots\\
  q^1_N&\cdots&q^d_N\\
\end{pmatrix}.
$
In particular, the row sums may vary significantly: Consider $d=2$ and
the uniform discrete distribution on $\{0,\dots,N\}$. If $L_1$ and
$L_2$ are comonotonic (i.e.~there is perfect positive dependence among
the random variables), then $(q^1_0,\dots,q^1_N)=(q^2_0,\dots,q^2_N)$
with row sums $\{0,2,\dots,2N\}$. If, on the other hand, $F_1$ and
$F_2$ are countermonotonic (perfect negative dependence among the
random variables), then $(q^1_0,\dots,q^1_N)=(q^2_N,\dots,q^2_0)$, and
all row sums are equal to $N$.  If we want to find an upper bound for
$F_L^{-1}(\alpha)$ we need to consider matrices with entries $q^j_r$
for $\tfrac{r}{N}\geq\alpha$, and for lower bounds matrices
constructed from $q^j_r$ for $\tfrac{r}{N}\leq\alpha$ and each time
minimize the variance of the row sums of $A$. This intuition is made
exact by a representation theorem of Rüschendorf~\cite[Theorem
2]{rueschendorf:83b}, showing that for discrete distribution
functions, and due to the uniform discretization inherent in our
definition of $\underline{F_j}$ and $\overline{F_j}$, solving the
minimum variance problem amounts to determining $\gamma(A)-\beta(A)$
for the matrix $A$, since it is enough to minimize over the set of all
rearrangements of the $F_j$.  We refer
to~\cite{puccetti-rueschendorf:11,puccetti-wang-wang:12,wang-wang:11,puccetti-rueschendorf:11b}
for recent applications and to~\cite{rueschendorf:83b}
and~\cite{rueschendorf:81,day:70:thesis} for more details on the
general concept of rearrangements of functions.

\begin{example}[\cite{embrechts-puccetti-rueschendorf:12}]
  Under the Basel II and III regulatory framework for banking
  supervision, large international banks are allowed to come up with
  internal models for the calculation of risk capital. For operational
  risk the so-called Loss Distribution Approach gives them full
  freedom concerning the stochastic modeling assumptions used. The
  resulting risk capital must correspond to a 99.9\%-quantile of the
  aggregated loss data over a year. This corresponds to computing the
  Value-at-Risk $\VaR_{0.999}(L)$ at $\alpha=0.999$ for an aggregate
  loss random variable $L=\sum_{i=1}^dL_i$, but makes no requirements
  on the interdependence between the individual loss random
  variables $L_i$ corresponding to the indivdual business lines:
  Assumptions made in the calculation must only be plausible and well
  founded.  Estimating the upper bound and lower bound of the VaR over
  all possible dependence structures is hence relevant both from the
  regulator's point of view, as well as from the bank's point of view,
  to estimate worst case hidden risks in the models presented under
  the Loss Distribution Approach.
\end{example}

Besides computing (or approximating) $\gamma(A)$ and $\beta(A)$, one
is also interested in deciding whether for a given matrix
$\gamma(A)=\beta(A)$. We will call such a matrix \emph{completely
  mixable}, in analogy with the definition of this concept by Wang and
Wang~\cite{wang-wang:11} for distribution functions.

In this paper we show that deciding complete mixability is a strongly
$\NP$-complete problem, even for a fixed number of columns, but can be
solved using dynamic programming in pseudopolynomial time for a fixed
number of rows. We show that the algorithm proposed by Puccetti et
al. in~\cite{puccetti-rueschendorf:11} to compute $\gamma(A)$ and
$\beta(A)$ is not an exact method unless $\NP\subseteq\ZPP$, despite
its impressive computational
success~\cite{embrechts-puccetti-rueschendorf:12}. Finally, for
matrices in fixed (column) dimension we present a
polynomial-time approximation scheme.

\section{Complexity}
\label{sec:complexity}

It is known that for two columns the complete mixability problem is
solvable explicitly (see the references
in~\cite{rueschendorf:83}). This is also apparent by recognizing that
the computation of $\gamma(A)$ can be understood as solving a
multidimensional bottleneck assignment problem. The multidimensional
bottleneck assignment problem asks for the computation of
$$\min_{\pi_1,\dots,\pi_d} \max_{1\leq i\leq m} c_{\pi_1(i),\dots,\pi_d(i)}
$$
for a $\underbrace{m\times\dots\times m}_d$ cost table
$C$. Defining
$c_{i_1,\dots,i_d}=A_{i_1,1}+\dots+A_{i_d,d}$ we see that $\gamma(A)$
can be computed by solving a multidimensional bottleneck assignment
problem. Using Observation~\ref{obs:beta-and-gamma} below we can
similarly compute $\beta(A)$ and thus check complete mixability.

In dimension $2$, the bottleneck assignment problem models the
following problem: Given a set of workers and a set of tasks, where the time of
worker $i$ performing task $j$ is $c_{ij}$, find a simultaneous
assignment of all workers to all tasks such that the maximal time
spent by any worker (the bottleneck of the schedule) is
minimized. Fulkerson et al. showed that the $2$-dimensional bottleneck
assignment problem can be transformed into a linear assignment
problem~\cite{fulkerson-glicksberg-gross:53}, and thus is polynomially
solvable.

The multi-dimensional bottleneck assignment problem of assigning
(equal-sized) crews of workers to (equal-sized) groups of tasks is
much harder. Even restricted versions of the 3-dimensional version do
not admit a polynomial time approximation
scheme~\cite{goossens-polyakovskiy-spieksma-woeginger:10}.

By adding $\mu=-\min_{1\leq i\leq m,\\1\leq j\leq d} A_{ij}$ to each
entry of $A$ we can always shift the matrix to make the smallest entry
equal to zero, changing all row sums by $+\mu\cdot d$. For convenience
we will hence restrict our attention to integral, nonnegative
matrices.  Assuming integrality is not a major restriction, since
rational matrices can without loss of generality be scaled to become
integral, and rational matrices provide a dense subset of the real
matrices that could arise in discretizing distribution functions.

First note that $\beta$ and $\gamma$ are related as follows:
\begin{observation}\label{obs:beta-and-gamma}
  Let $A\in\Z^{m\times d}$, and
  $l:=\max_{1\leq i\leq m,1\leq j\leq d}A_{ij}$ its largest
  entry. Define $A'$ by $A'_{ij}=l-A_{ij}$. Then $\beta(A)=d\cdot l -\gamma(A')$.
\end{observation}

Hence we only ever need to consider one of the two values. To see that
deciding complete mixability of $A$ and computing $\beta$ or $\gamma$
are actually polynomially equivalent we only need the following
obvious necessary condition that will also prove useful later on.

\begin{observation}\label{obs:rowsum}
  Let $A\in\Z^{m\times d}$. $A$ is completely mixable if and only if
  $\gamma(A)=\beta(A)=\tfrac{1}{m}\sum_{i=1}^m\sum_{j=1}^d A_{ij}$.
\end{observation}

It turns out that this is sufficient for showing linear time
decidability of complete mixability if the entries of $A$ are
restricted to at most two values: Those can be mapped to $\{0,1\}$,
and then the algorithm used in the proof below provides a linear time
check for complete mixability:

\begin{theorem}\label{lem:01cm}
  Let $A\in\{0,1\}^{m\times d}$. A is completely mixable if and only
  if $m\mid \sum_{1\leq i\leq m,1\leq j\leq d} A_{ij}$. The permutation achieving the
  complete mix can be computed in linear time $\Oh(m\cdot d)$.
\end{theorem}
\begin{proof}
  ``$\Rightarrow$'' Let $s=\sum_{i=1}^m\sum_{j=1}^d A_{ij}$. If $m \nmid s$ then $A$
  cannot be completely mixable.

  ``$\Leftarrow$'' Assume $m\mid \sum_{1\leq i\leq m,1\leq j\leq d}
  A_{ij}$. We need to permute the
  columns of $A$ such that exactly $\tfrac{s}{m}\in\{0,\dots,d\}=r$
  entries in each row have value $1$.
  
  This can always be done: Define for $i\in\{1,\dots,m\}$
  the \emph{defect} $\delta(i)=r -
  \sum_{j=1}^dA_{ij}$ and $\phi=\sum_{i=1}^m|\delta(i)|$ the total
  defect. Clearly, $\phi=0$ if and only if all row sums of the matrix
  are equal to $r$.

  Starting with $j=2$ define 
  $S_j=\{i\in\{1,\dots, m\}\;:\; \delta(i)>0, A_{ij}=1\}$
  and 
  $D_j=\{i\in\{1,\dots,m\}\;:\; \delta(i)<0,A_{ij}=0\}.$
  If $S_j\neq\emptyset$ and
  $D_j\neq\emptyset$ let $t_j=\min\{|S_j|,|D_j|\}$ and swap the entries of
  column $A_{\cdot j}$ indexed by the largest $t_j$ entries of $S_j$
  with those indexed by the smallest $t_j$ entries of $D_j$. Repeat
  in increasing order, for all $j\leq d$.

  Clearly, throughout the procedure the defect of rows with positive
  defect can only decrease, and the defect of rows with negative
  defect can only increase; the total defect decreases by $2t_j>0$ for
  each swap. Assume that the procedure stops in the last column with a
  matrix that has nonzero total defect $\phi$. Then there must be a
  row $i_1$ with positive defect $\delta_{i_1}$ and a row $i_2$ with
  negative defect $\delta_{i_2}$, since $r=s/m$. Consider some column
  index $l$ such that $A_{i_1 l}=1$ and $A_{i_2 l}=0$. Then the index
  $i_1$ was in $S_l$, and $i_2$ was in $D_l$ (because the absolute
  defects of the rows can only have decreased in later steps), but
  they were not swapped, a contradiction.
\end{proof}

Note that when the algorithm declares $A$ `not completely mixable', it
has computed a permutation achieving maximal row sum.  

We note in passing that if $A\in\Z^{m\times d_1}$ and $B\in\Z^{m\times
  d_2}$ are completely mixable, then so is $-A$ and $(A
B)\in\Z^{m\times (d_1+d_2)}$. A more interesting composition is the
following:
\begin{proposition}[glueing of completely mixable matrices]\label{prop:glueing}
  Let $A\in\R^{m_1\times d_1}$ and $B\in\R^{m_2\times d_2}$ be
  completely mixable matrices that have been permuted to each have
  equal row sums. Then the matrix 
  $$\textstyle
  A\oplus B = (C_{ij})_{{1\leq i \leq m_1m_2}\atop{1\leq j\leq d_1d_2}}
  $$
  with $C_{m_2(i-1)+k,d_2(j-1)+l}=A_{ij}+B_{kl}$ (i.e., the block
  matrix constructed by replacing every entry $A_{ij}$ of $A$ by a
  block $(A_{ij}+B_{kl})_{{1\leq k\leq m_2}\atop{1\leq l\leq d_2}}$)
  is completely mixable.
\end{proposition}
\begin{proof}
  Since $A$ and $B$ have identical row sums $\sigma_A$ and $\sigma_B$
  (respectively), the row sum of $C$ is always
  $d_2\cdot\sigma_A+d_1\cdot\sigma_B$, showing complete mixability of
  $C$.
\end{proof}

\smallskip
In general checking complete mixability is hard:
\begin{theorem}\label{thm:np-completeness}
  It is strongly $\NP$-complete to decide whether an integral matrix
  $A\in\Z^{m\times d}$ is completely mixable. It remains strongly
  $\NP$-complete for fixed $d$, and at least weakly $\NP$-complete for
  fixed $m$.
\end{theorem}

\begin{proof}
  Even for $d=3$ we are looking at a \textsc{Numerical 3-dimensional
    Matching} problem, which is strongly $\NP$-complete~\cite[problem
  \textsc{SP16}]{GareyJohnson79} (the row sum that needs to be tested
  is given by Observation~\ref{obs:rowsum}).

  For $m=2$ we can reduce \textsc{Number Partition} to
  this problem: Let $(n_1,\dots,n_d)\in\Z^d$ be a multiset of
  integers, and let $s=\sum_i n_i$. Then $A=
  \begin{pmatrix}
    n_1 & \dots & n_d\\
    0   & \dots & 0\\
  \end{pmatrix}$ is completely mixable if and only if
  $(n_1,\dots,n_d)$ can be partitioned into two multisets of equal
  size $\tfrac{1}{2}s$. This problem is known to be (weakly)
  $\NP$-complete~\cite{garey-johnson:78}.
\end{proof}


We note that, as is the case for many $\NP$-hard problems, there can
not be a polynomial time approximation algorithm computing an
approximate value $\gamma'(A)$ that achieves an additive error
$|\gamma'(A)-\gamma(A)|\leq K$ for some constant $K$: For a given
completely mixable matrix $A\in\Z^{m\times d}$ the matrix obtained by
appending the column $(K',2K',\dots,mK')^\top$ with
$K'\geq\max{2da^*,K}$ (where $a^*$ denotes the largest entry of $A$)
has all row sums separated by at least $K'$, so approximating
$\gamma(A)$ to within $K$ amounts to deciding complete mixability.

Clearly, when both $d$ and $m$ are fixed the problem is trivial by
enumeration. For fixed $m$ and variable $d$ a dynamic programming
algorithm similar to the one for \textsc{Number Partition} of Garey
and Johnson~\cite{GareyJohnson79} can be devised to check complete
mixability:
\begin{lemma}
  There is a pseudopolynomial algorithm to decide complete mixability
  for matrices $A\in\Z_{\geq0}^{m\times d}$ if the number of rows $m$ is fixed.
\end{lemma}
\begin{proof}
  We can enumerate all possible values appearing as row sums as
  $v_1,\dots,v_N$, with $N\leq
  d\cdot\max_{1\leq i\leq m,1\leq j\leq d}A_{ij}$.  Build a
  dynamic programming table $B$ with Boolean entries $B(i,j,v_1,\dots,v_N)$, where
  $B(i,j,\dots,r,\dots)$ is \textsc{True} if and only if value $r$
  can be constructed as a (partial) row sum in row $i$ with $j$
  columns: Iterate over the columns of $A$ succesively and update $B$
  using each of the (fixed number of) permutations that can be applied
  to column $j$ of $A$. Then $A$ is completely mixable if
  $B(i,d,\dots,r,\dots)$ is \textsc{True} for all rows $i$, where $r$
  is the target row sum $\tfrac{1}{m}\sum_{i,j}A_{ij}$.
\end{proof}

The results in~\cite{goossens-polyakovskiy-spieksma-woeginger:10} for
the bottleneck $3$-assignment problem with costs defined by distances
(\textbf{B3AP-per}) yield a $2$-approxi\-mation for determining
$\gamma(A)$ and $\beta(A)$.

\begin{lemma}
  For $A\in\Z_{\geq0}^{m\times 3}$ there exists a polynomial
  2-approxi\-mation algorithm
  for computing $\gamma(A)$.
\end{lemma}

\begin{proof}
  For convenience we will in this proof assume that the matrix $A$ is
  indexed by $(i,j)$ with $0\leq i\leq m-1$ and $0\leq j\leq d-1$.
  We construct an instance of \textsc{B3AP-per} as follows: Let
  $I=\{0,\dots,3m-1\}$ denote the indices of all elements of $A$ in
  column-major order, i.e.~index $l\in I$ selects element $(\lfloor
  \tfrac{l}{3}\rfloor,l\mod 3)$ of $A$, and define the sets
  $R=\{3k+1\;|\; k<m\}$, $G=\{3k+2\;|\;k<m\}$, and
  $B=\{3k+3\;|\;k<m\}$ such that $I=R\cup G\cup B$. Define
  $\dist(i,j)=\tfrac{1}{2}(A_{(\lfloor\tfrac{i}{3}\rfloor),(i\mod
    3)}+A_{(\lfloor\tfrac{j}{3}\rfloor),(j\mod 3)})$. Then $\dist$
  satisfies the triangle inequality and is symmetric. It does not
  necessarily satisfy $\dist(i,i)=0$, so is not a proper
  metric. Nevertheless, Theorem~1
  of~\cite{goossens-polyakovskiy-spieksma-woeginger:10} holds with the
  original proof, as only symmetry and triangle inequality are
  exploited, and $\dist(i,j)$ is only ever evaluated between pairs of
  different index sets from $\{R,G,B\}$, i.e.
  $\lfloor\tfrac{i}{3}\rfloor\neq\lfloor\tfrac{j}{3}\rfloor$.
  With
  our definition of $\dist(i,j)$
  \[
  \begin{array}{rl}
    c_{ijk}&=\dist(i,j)+ \dist(j,k)+\dist(k,l)\\
          &=\tfrac{1}{2}\bigl((A_{i1}+A_{j2})+(A_{j2}+A_{k3})+(A_{k3}+A_{i1})\big)

  \end{array},
  \]
  since costs need only be defined for $i\in R,j\in G,k\in B$. Then
  determining $\gamma(A)$ is exactly the \textsc{B3AP-per} problem
  of~\cite{goossens-polyakovskiy-spieksma-woeginger:10}.
\end{proof}

\section{The swapping algorithm}
\label{sec:swapping-algorithm}

As noted by Puccetti and Rüschendorf~\cite{puccetti-rueschendorf:11},
it is sometimes easy to check that a matrix can be permuted so as to
increase its smallest row sum. We need the following definition:

\begin{definition}
  For $A\in\Z^{m\times d}$ let $\dropped{A}{j}$ denote the matrix
  obtained from $A$ by dropping its $j$-th column,
  i.e. $\dropped{A}{j}=(A_{\cdot 1}\dots A_{\cdot (j-1)}A_{\cdot
    (j+1)}\dots A_{\cdot d})$.

  For $x,y\in\Z^m$ denote by $x\antisorted y$ that $x$ and $y$ are
  \emph{oppositely ordered}, i.e.
  there exists a permutation $\pi\in\Symm(m)$ such that
  $x_{\pi(1}\leq\dots\leq x_{\pi{m}}$ and $y_{\pi(1)}\geq\dots\geq y_{\pi(m)}$.
\end{definition}

\begin{lemma}[Theorem 3.1
  of~\cite{puccetti-rueschendorf:11}]\label{lem:swapping-algorithm}
  Let $A\in\Z^{m\times d}$. If there exists a column index $j$
  such that $(\sum_l \dropped{A}{j}_{1l},\dots,\sum_l \dropped{A}{j}_{ml})^\top \notantisorted A_{\cdot j}$, then column $A_{\cdot j}$ can be permuted such that
  opposite ordering is achieved, and the minimal row sum of $A$ does not
  decrease.
\end{lemma}
For completeness we give the following proof.
\begin{proof}
  Let $(\sum_l \dropped{A}{j}_{1l},\dots,\sum_l
  \dropped{A}{j}_{ml})^\top =: x \notantisorted y:=A_{\cdot j}$. Then
  there exists a pair of indices $i_1,i_2$ such that $x_{i_1}\leq
  x_{i_2}$ and $y_{i_1}\leq y_{i_2}$. Therefore $x_{i_1}+y_{i_1} \leq
  x_{i_1}+y_{i_2}$ and $x_{i_1}+y_{i_1} \leq
  x_{i_2}+y_{i_1}$. Hence
  $$\min\{x_{i_1}+y_{i_1},x_{i_2}+y_{i_2}\}=x_{i_1}+y_{i_1}
  \leq\min\{x_{i_1}+y_{i_2},x_{i_2}+y_{i_1}\},$$
  and thus swapping $y_{i_1}\leftrightarrow y_{i_2}$ cannot decrease
  the minimal row sum of $A$.

  We note that if both $x_{i_1}< x_{i_2}$  and $y_{i_1}< y_{i_2}$, and
  there are no duplicate entries in $x$ and $y$,
  then the minimal row sum of $A$ will actually increase by at least
  $1$ if $i=\argmin_{i_1,i_2}\{x_{i_1}+ y_{i_1},x_{i_2}+ y_{i_2}\}$ is chosen
  minimally.
\end{proof}

In~\cite{puccetti-rueschendorf:11} this is taken as a rationale to
propose the following algorithm:

\begin{algorithm}
  \caption{Swapping Algorithm}\label{alg:swapping}
  \begin{algorithmic}[1]
    \Procedure{AntisortColumns}{$A$}
     \While{$\exists j: (\sum_l \dropped{A}{j}_{1l},\dots,\sum_l
       \dropped{A}{j}_{ml})^\top \notantisorted A_{\cdot j}$}
      \State $x\gets (\sum_l \dropped{A}{j}_{1l},\dots,\sum_l
       \dropped{A}{j}_{ml})^\top$
      \State $y\gets A_{\cdot j}$
      \State select $(i_1,i_2)$ from $\{(i_1,i_2)\;|\;
      x_{i_1}<x_{i_2}\wedge y_{i_1}<y_{i_2}\}$
      \State swap $A_{i_1 j}\leftrightarrow A_{i_2 j}$
     \EndWhile
    \EndProcedure
  \end{algorithmic}
\end{algorithm}

It is then stated and confirmed experimentally that running this
algorithm on many randomly permuted copies of the matrix $A$ will
usually determine very good bounds for $\beta(A)$ and $\gamma(A)$, and
is often very fast. In~\cite{embrechts-puccetti-rueschendorf:12} it is
admitted that no analytic proof of convergence to the optimum is
known, even when randomly permuting the starting matrix, despite the
promising practical results.  This is to be expected:

\begin{lemma}
  The swapping algorithm~\ref{alg:swapping}
  of~\cite{puccetti-rueschendorf:11} does not run in expected
  polynomial time unless $\NP\subseteq\ZPP$.
\end{lemma}
\begin{proof}
  Consider an instance of the complete mixability problem. Apply the
  swapping algorithm. Assume that the expected number of times that
  the input matrix has to be randomly permuted before the swapping
  algorithm correctly decides complete mixability were of polynomial
  size. Since we have shown in Theorem~\ref{thm:np-completeness} that
  the problem is strongly $\NP$-complete this would yield a zero-error
  probabilistic polynomial time algorithm~\cite{gill:77} for all problems in
  $\NP$. This would imply $\NP\subseteq\ZPP$.
\end{proof}

In fact, the algorithm may terminate with an approximation error of
$\Oh(\max_{ij}A_{ij})$ (Lemma~\ref{lem:swapping-bad}).

For some matrices, however, Lemma~\ref{lem:swapping-algorithm}
actually guarantees a positive increase of the minimal row sum: As noted at
the end of the proof of Lemma~\ref{lem:swapping-algorithm}, swapping
entries in a column, say $j$, to achieve opposite ordering will
actually increase the minimal row sum by at least $1$, unless there
are duplicate entries in $j$ or duplicate row sums in the matrix
$\dropped{A}{j}$. This yields
\begin{observation}\label{obs:swapping-pseudopoly}
  Let $A\in\Z^{m\times d}$ be a matrix where all columns have $m$
  different entries, and for which all $(d-1)$-column submatrices
  obtained by deleting a single column have the property that for all
  possible permutation of column entries their $m$ row sums have $m$
  distinct values. Then $\gamma(A)$ and $\beta(A)$ can be determined
  in pseudopolynomial time using the swapping algorithm.
\end{observation}

It is not unlikely that a matrix with entries drawn uniformly at
randomly from a large domain with few rows has no duplicate row sums
(Lemma~\ref{lem:duplicate-row-sums}), but it seems very hard to trace
how this probability evolves after a few steps of swapping.

\begin{lemma}\label{lem:duplicate-row-sums}
  Let $A\in\Z^{m\times (d+1)}_{\geq0}$ be a matrix where  each column
  contains $m$ entries drawn uniformly at random from $\{1,\dots,N\}$.
  Then the probability $p_{\neq}(A)$ for a $d$-column submatrix of
  $A$ to have $m$ distinct row sums is
  $$p_{\neq}(A)\geq 1-\Oh(\tfrac{m^2}{N}).$$
\end{lemma}

\begin{proof}
  Consider a $d$-column submatrix $M$. Each entry of $M$ is a random
  variable, independently drawn from $\{1,\dots,N\}$. We consider the
  entries of $M$ drawn from ${1,\dots,N}$ row by row. Hence the
  probability of obtaining sum $s$ in one row is\linebreak[4]
  $\Prob{M_{i1}+\dots+M_{id}=s}=\tfrac{p(s,d)}{N^d}$, where $p(s,d)$ is
  the number of partitions of $s$ into exactly $d$ parts. The probability of
  not obtaining sum $s$ is $\tfrac{N^d-p(s,d)}{N^d}$.
  
  Matrix $M$ has $m$ rows; using the binomial distribution formula the
  probability of obtaining sum $s$ in $m$ one-row trials is thus
  \begin{equation*}
  \begin{array}{l}
  \Prob{\text{row sum $s$ at least twice in $M$}}\\
  =1
  -\left(\frac{N^d-p(s,d)}{N^d}\right)^m
  -m\frac{p(s,d)}{N^d}\left(\frac{N^d-p(s,d)}{N^d}\right)^{m-1}.
\end{array}
\end{equation*}
  Therefore the probability for $M$ to have distinct row sums is
  \begin{align*}
    &\Prob{M\text{ has distinct row sums}} \\
    &=1-\Prob{\exists s:\text{ row sum $s$ at least twice in $M$}}\\
\intertext{and since we can have at most $m$ row sums,}
    &\geq 1-m\Prob{\text{most likely duplic. rowsum $s^*$ at least twice
      in $M$}}\\
    &=1-m\left(
      1
      -\left(\frac{N^d-p(s^*,d)}{N^d}\right)^m\right.\\
    &\hphantom{1-m}\left.
      -m\frac{p(s^*,d)}{N^d}\left(\frac{N^d-p(s^*,d)}{N^d}\right)^{m-1}
      \right)\\
\intertext{where, to upper bound the probability of duplicates, we need to lower bound $(N^d-p(s^*,d))$}
    &\geq 1-m\left(
      1
      -\frac{(N^d-N^{d-1})^{m}}{N^{dm}}\right.\\
      &\hphantom{1-m}
      \left. -m p(s^*,d)
      \frac{(N^d-N^{d-1})^{m-1}}{N^{dm}}
      \right)\\
    &\geq 1-m\left(
      1
      -\frac{\Oh(N^{dm}+(-1)^{m-1}N^{d(m-1)})}{N^{dm}}\right.\\
    &\hphantom{1-m}\left.
      -m\cdot 1\cdot
      \frac{\Oh(N^{d(m-1)}+(-1)^{m-2}N^{d(m-2)})}{N^{dm}}
      \right)\\
    &\geq 1-\Oh(\tfrac{m^2}{N})
  \end{align*}
  where for the partition of $s$ into $d$ parts we use the trivial
  lower bound of $1$ and the
  generous upper bound $p(s,d)\leq (s-d+1)^{d-1}$ which is obtained as
  follows: To partition $s^*$ we need to use at least $1$ unit in each
  of the $d$ parts. We now still can distribute $s^*-d$ units into $d$
  bins; we can choose freely from $\{0,\dots,s^*-d\}$ for $d-1$ bins,
  then the amount for the last bin is determined.
\end{proof}

  

\section{Matrices of consecutive integers}

\begin{definition}
  Let $d,N\in\Z_{\geq0}$ and $a=(1,\dots,N)^\top$. Every matrix
  $A^\Pi$ obtained through permutations $\Pi\in\Symm(N)^d$ of the columns
  from $A=(a,\dots,a)\in\Z^{N\times d}$ will be called
  \emph{$(N,d)$-complete consecutive integers matrix}.
\end{definition}

We will now show that for such matrices and certain choices of $N$
(given $d$) the values of $\beta$ and $\gamma$ can be computed
explicitly, and that these yield bounds for arbitrary values of $N$.
Furthermore, we will demonstrate that the swapping algorithm
of~\cite{puccetti-rueschendorf:11} (Algorithm~\ref{alg:swapping}) on
these instances does not have a constant factor approximation
guarantee (it is at least $\Oh(N)$).

\begin{theorem}
  Let $A\in\Z^{N\times d}_{\geq0}$ be a $(N,d)$-complete consecutive integers matrix and
  $N=d^k$ for some $0<k\in\Z$. Then $A$ is completely mixable and 
  $$\gamma(A)=\beta(A)=d+\sum_{i=0}^{d-1}\sum_{j=1}^{k} i\cdot d^{j-1}=:a_d(k).$$
\end{theorem}

\begin{proof}
  For $k=1$ the matrix $A=
  \begin{pmatrix}
    1     & 2      &\dots&d\\
    \vdots& \vdots &\iddots&\colon\\
    d-1   & d      &     &d-2\\
    d     & 1      &\dots&d-1
  \end{pmatrix}$ is a permutation that shows that the $(d,d)$-complete consecutive
  integers matrix is completely mixable with uniform row
  sum $\sum_{i=0}^d
  i=d+\sum_{i=0}^{d-1}i=d+\sum_{i=0}^{d-1}i\sum_{j=1}^1d^0=a_d(1)$.

  Assume that the statement holds for $k\in\Z_{\geq0}$, i.e.~that a
  $(d^k,d)$-complete consecutive integers matrix $A$ of size
  $d^k\times d$ has been reordered into a matrix $A'$ with identical
  row sums $a_d(k)$.  We will use $A'$ to construct a matrix $A''$
  with $d^{k+1}$ rows that is a reordering of the
  $(d^{k+1},d)$-complete consecutive integers matrix of size $d^{k+1}$
  and has row sums $a_d(k+1)$: We use the glueing operation of
  Proposition~\ref{prop:glueing} between $A'$ and $B= d^k\begin{pmatrix}
    0      & 1       & \dots   & d-1\\
    1      & 2       & \iddots & 0  \\
    \vdots & \iddots & \iddots & \vdots\\
    d-1 & 0 & \dots & d-2
     \end{pmatrix}$ (which has constant row sum $d^k
     \tfrac{d(d-1)}{2}$), to obtain $A''=A'\oplus B$, which has row
     sum $a_d(k)+\sum_{0\leq i<d}i\cdot d^k=a_d(k+1)$.
\end{proof}


\begin{corollary}
  Let $A\in\Z^{N\times d}_{\geq0}$ be a $(N,d)$-complete consecutive integers
  matrix. Then
  $$a_d(\lfloor\log_d(N)\rfloor) \leq \beta(A)\leq
  \gamma(A)\leq a_d(\lceil\log_d(N)\rceil).$$
  
  In particular, by underestimating $\beta(A)$ as
  $a_d(\lfloor\log_d(N)\rfloor)$ and overestimating $\gamma(A)$ as
  $a_d(\lceil\log_d(N)\rceil)$ we make an additive error of at
  most $\sum_{i=0}^{d-1}i\cdot d^{\lceil\log_d(N)\rceil-1}$ (which
  is roughly $\tfrac{d^2N}{2}$).
\end{corollary}

\begin{lemma}\label{lem:swapping-bad}
  Let $A$ be a $(N,3)$-complete consecutive integers matrix where all
  permutations are the identity. Then the swapping algorithm will
  terminate after one reordering step with a matrix with row sums in the
  range of $[N+2,\dots,2N+1]$. In particular, if $N=3^k$ and $A$ is
  hence completely mixable the solution is has additive error $\Oh(N)$.
\end{lemma}

\begin{proof}
  Starting with $A=
    \begin{pmatrix}
      1      &1 & 1\\
      \vdots &\vdots&\vdots\\
      N      &N & N
    \end{pmatrix} $ the swapping algorithm will invert the order of
    the first column to obtain
    $A'=
    \begin{pmatrix}
      N      & 1     & 1\\
      \vdots &\vdots &\vdots\\
      1      & N     & N
    \end{pmatrix}$.
    This matrix satisfies the rule that each column is sorted
    anti-monotonously wrt. the sums of the other two columns, so the
    algorithm stops. The row sums are $N+2,N+3,\dots,2N,2N+1$.

    Since for $N=3^k$ we know that there exists a reordering of $A$
    such that all row sums are $3+\sum_{i=1}^k3^i$ this shows an
    approximation error of at least $\Oh(N)$.
\end{proof}

\section{Matrices with restricted domain}
\label{sec:matr-with-restr}

Matrices of consecutive integer entries are just a special case of
matrices where all columns contain the same multiset of entries
$M=\{v_1,\dots,v_m\}$. If the number of different entries in $M$ is
fixed, these matrices yield tractable instances for variable $d$, much
like an $N$-fold system.

\begin{lemma}
  Let $A\in\Z^{m\times d}$ such that the entries of each column come
  from the same multiset $M=\{a_1,\dots,a_m\}$, and assume $m$ is
  fixed. Then $\gamma(A)$ can be computed in polynomial time.
\end{lemma}
\begin{proof}
  Since the multiset $M$ is fixed,
  there are only a fixed number of
  different ways to rearrange a column by permutations. For each of
  these $k$ arrangements of the set $M$ denote the 
    permutation by $\pi_l$, $1\leq l\leq k$.  Then 
    $$
    \begin{array}{rl}
          \min \Gamma&\\
          \begin{pmatrix}
            v_{\pi_1(1)}\\
            \vdots\\
            v_{\pi_1(m)}
          \end{pmatrix} x_{\pi_1}
          + \dots +
          \begin{pmatrix}
            v_{\pi_{k}(1)}\\
            \vdots\\
            v_{\pi_{k}(m)} 
          \end{pmatrix} x_{\pi_{k}}
          &\leq
          \begin{pmatrix}
            \Gamma\\
            \vdots\\
            \Gamma
          \end{pmatrix}\\
          \sum_{l=1}^k x_{\pi_{l}} &= d\\
          x_{\pi_{l}}\in \Z_{\geq0}&\text{ for $1\leq l\leq k$}
        \end{array}
        $$
        is an integer programming problem in fixed dimension $k$,
        modeling that we have to choose $d$ rearrangements of the set
        $M$ (one for each column of $A$) that can be solved in
        polynomial time~\cite{Lenstra83}.
\end{proof}

Instead of instances with the same multiset of values in every column we
can also consider instances where all matrix entries come from a fixed
set of values, generalizing the two-value case of
Lemma~\ref{lem:01cm}. 
\begin{theorem}\label{lem:fixed-M-d}
  Let $M=\{v_1,\dots,v_s\}\subseteq\R$ be a fixed set of values and $A\in M^{m\times
    d}$. For every fixed number of columns $d$ one can compute
  $\gamma(A)$ in polynomial time. 
\end{theorem}
\begin{proof}
  If $M$ is fixed then for fixed $d$ there are at most $s^d$ possible
  row vectors  $r_1,\dots,r_{s^d}$ composed of values from $M$.
  We define the binary value $u_{ij}^k$
  to be $1$ if and only if $(r_k)_j=v_i$, i.e.~if in row vector $k$
  the value $v_i$ appears in the $j$-th column.

  For a given matrix $A\in M^{m\times d}$ we can count the number of
  occurences of value $v_i$ in column $j$ in polynomial time. Denote
  these values by $o_{ij}$. 

  Introduce binary variables $p_1,\dots,p_{s^d}$ to indicate whether
  pattern $k$ occurs in the permuted version of $A$, and integer
  variables $q_1,\dots,q_{s^d}$ counting how often it appears.
  Then the following integer program in fixed dimension $s^d$ can be
  used compute $\gamma(A)$:
  $$
  \begin{array}{rl@{\qquad\qquad}l}
    \min \Gamma&\\
    (\sum_{j=1}^d (r_k)_j) p_k &\leq \Gamma&\text{for all $k$}\\
    \sum_{k=1}^{s^d} p_k&\leq m\\
    p_i&\leq q_i                         &\text{for all $k$}\\
    q_i&\leq m p_i                       &\text{for all $k$}\\
    \sum_{k=1}^{s^d} u_{ij}^k q_k &= o_{ij} &\text{for all $i,j$}\\
    &p_i\in\{0,1\}^{s^d},q_i\in Z^{s^d}
  \end{array}
  $$
\end{proof}

\begin{corollary}
  There exists a polynomial approximation scheme for every fixed $d$
  to compute $\gamma(A)$ for $A\in\R^{m\times d}\geq 0$ with multiplicative
  error $(1+\epsilon)$ for every $\epsilon>0$.
\end{corollary}
\begin{proof}
  Define a grid of width $\epsilon \tfrac{a^*}{d}$ where $a^*$ is the
  largest entry of $A$. Consider the set $M=\{0,\epsilon
  \tfrac{a^*}{d},2\epsilon
  \tfrac{a^*}{d},\dots,\lceil\tfrac{d}{\epsilon}\rceil\epsilon
  \tfrac{a^*}{d}\}$ and round the entries of $A$ up to next value in
  $M$ to obtain an approximating instance $\bar{A}$. Then by
  Lemma~\ref{lem:fixed-M-d} the approximating instance can be solved
  in polynomial time since $M$ has $\lceil\tfrac{d}{\epsilon}\rceil+1$
  entries, a number only depending on the fixed $d$ and $\epsilon$.
  The objective value of the approximate solution is at most
  $d\epsilon \tfrac{a^*}{d}\leq\epsilon \gamma(A)$ larger than
  $\gamma(A)$, since $a^*\leq\gamma(A)$, yielding a
  $(1+\epsilon)$-approximation.
\end{proof}

\section*{Acknowledgments}

The author wants to thank Giovanni Puccetti for bringing the question
to his attention, and David Adjiashvili, Robert Weismantel and Sandro
Bosio for helpful discussions.\\
Part of this research was supported by EU-FP7-PEOPLE project 289581 `NPlast'.

\section*{References}
\label{sec:references}

\bibliographystyle{elsarticle-num} 
\bibliography{weismantel,finance}


  
  
\end{document}